\documentclass[11pt,reqno]{amsart}
\usepackage[ruled,noline,linesnumbered]{algorithm2e}
\SetStartEndCondition{ }{}{}%
\SetKwProg{Fn}{def}{\string:}{}
\SetKwFunction{Range}{range}%%
\SetKw{KwTo}{in}\SetKwFor{For}{for}{\string:}{}%
\SetKwIF{If}{ElseIf}{Else}{if}{:}{else if}{else:}{}%
\SetKwFor{While}{while}{:}{}%
\AlgoDontDisplayBlockMarkers\SetAlgoNoEnd\SetAlgoNoLine%
\SetKwInOut{Input}{Input}
\SetKwInOut{Output}{Output}

\usepackage{amssymb, amsmath, amsthm,mathrsfs,calligra}
\usepackage[foot]{amsaddr}
\usepackage{hyperref}
\hypersetup{
	pdfauthor   = {},
	pdftitle    = {Two-cover descent on plane quartics with rational bitangents},
	pdfsubject  = {},
	pdfkeywords = {},
	backref=true, pagebackref=true, hyperindex=true, colorlinks=true,
	breaklinks=true, urlcolor=blue, linkcolor=blue, citecolor=blue,
	bookmarks=true, bookmarksopen=true}
\usepackage[alphabetic,backrefs,lite]{amsrefs}
\usepackage[height=9in,width=6in,marginparwidth=0.8in]{geometry}
\usepackage{tikz-cd}
\usepackage{setspace}
\usepackage{mathtools}
\usepackage{marginnote}
\usepackage{multirow}

\DeclareFontEncoding{OT2}{}{} % to enable usage of cyrillic fonts
\newcommand{\textcyr}[1]{%
 {\fontencoding{OT2}\fontfamily{wncyr}\fontseries{m}\fontshape{n}\selectfont #1}}
\newcommand{\Sha}{{\mbox{\textcyr{Sh}}}}
\newtheorem{lemma}{Lemma}[section]

\newtheorem{prop}[lemma]{Proposition}

\newtheorem{question}[lemma]{Question}
\newtheorem{claim*}{Claim}

\newtheorem{observation}[lemma]{Observation}

\theoremstyle{definition}
\newtheorem{remark}[lemma]{Remark}

\newtheorem{example}[lemma]{Example}
\newtheorem{defn}[lemma]{Definition}

\newcommand{\PP}{{\mathbb P}}
\newcommand{\CC}{{\mathbb C}}
\newcommand{\FF}{{\mathbb F}}
\newcommand{\QQ}{{\mathbb Q}}
\newcommand{\RR}{{\mathbb R}}
\newcommand{\ZZ}{{\mathbb Z}}

\newcommand{\kbar}{{k^\mathrm{alg}}}

\DeclareMathOperator{\sHom}{\mathscr{H}\text{\kern -4pt {\calligra\large om}}\,}

\newcommand{\bgamma}{\boldsymbol{\gamma}}
\newcommand{\btgamma}{\boldsymbol{\tilde{\gamma}}}
\newcommand{\calB}{{\mathcal B}}

\newcommand{\calO}{{\mathcal O}}
\newcommand{\calP}{{\mathcal P}}

\newcommand{\fp}{{\mathfrak p}}
\newcommand{\fq}{{\mathfrak q}}

\DeclareMathOperator{\HH}{H}

\DeclareMathOperator{\Aut}{Aut}

\DeclareMathOperator{\Gal}{Gal}

\DeclareMathOperator{\Sel}{Sel}

\DeclareMathOperator{\ord}{ord}
\DeclareMathOperator{\Sym}{Sym}
\DeclareMathOperator{\Div}{Div}
\DeclareMathOperator{\Pic}{Pic}
\DeclareMathOperator{\Jac}{Jac}

\DeclareMathOperator{\sep}{sep}

\DeclareMathOperator{\unr}{unr}

\DeclareMathOperator{\Sp}{Sp}

\renewcommand{\div}{\mathrm{div}}
\newcommand{\nospaceperiod}{\makebox[0pt][l]{\,.}}
\numberwithin{equation}{section}
\numberwithin{table}{section}

\title[Descent on plane quartics]{Two-cover descent on plane quartics with rational bitangents}

\author{Nils Bruin}
\address[1]{Department of Mathematics, Simon Fraser University, Burnaby, BC, Canada V5A 1S6}
\email{nbruin@cecm.sfu.ca}

\author{Daniel Lewis}
\address[2]{Department of Mathematics, The University of Arizona, 617 N. Santa Rita Ave, Tucson, AZ 85721, USA}
\email{dlewis3@math.arizona.edu}

\thanks{The first author acknowledges the support of the Natural Sciences and Engineering Research Council of
	Canada (NSERC), funding reference number RGPIN-2018-04191.}

\date{July 3, 2020}
\subjclass[2010]{Primary 11G30, 14H30; Secondary 11D41, 14H50}
\keywords{Plane quartics, rational points, local-to-global obstructions, bitangents, descent obstructions, two-covers}
\begin{document}

%%%%%%%%%%%%%%%%%%%%%%%%%%%%%%%%%%%%%%%%%%%%%%%%%%%%%%%%%%%%%%%%%%%%%%%%%%%%
\begin{abstract}
We implement two-cover descent for plane quartics over $\QQ$ with all 28 bitangents rational and show that
on a significant collection of test cases, it resolves the existence of rational points. We also review a classical description of the relevant moduli space and use it to generate examples. We observe that local obstructions are quite rare for such curves and only seem to occur in practice at primes of good reduction. In particular, having good reduction at $11$ implies having no rational points. We also gather numerical data on two-Selmer ranks of Jacobians of these curves, providing evidence these behave differently from those of general abelian varieties due to the frequent presence of an everywhere locally trivial torsor.
\end{abstract}
%%%%%%%%%%%%%%%%%%%%%%%%%%%%%%%%%%%%%%%%%%%%%%%%%%%%%%%%%%%%%%%%%%%%%%%%%%%%

\maketitle
\edef\marginnotetextwidth{2em}

\section{Introduction}

A central problem in arithmetic geometry is to determine if a variety $C$ over a number field $k$, for instance a nonsingular projective curve, has any $k$-rational points. The most elementary way of showing that $C(k)$ is empty is by showing that $C(k_v)=\emptyset$ for some completion $k_v$ of $k$. In that case, we say $C$ has a \emph{local obstruction} to having rational points.

We consider a more refined \emph{descent obstruction} here. Our construction can be read in elementary terms, but the theoretical motivation is enlightening. Suppose we have an unramified cover $\pi\colon D\to C$ of nonsingular proper varieties over $k$ with geometric automorphism group $\Gamma=\Aut_{\kbar}(D/C)$ satisfying $\#\Gamma=\deg(\pi)$. The \emph{twisting principle} \cite{milne:etale_cohomology}*{III.4.3(a)} gives us that the Galois cohomology set $\HH^1(k,\Gamma)$ parametrizes \emph{twists} $\pi_\gamma\colon D_\gamma\to C$, as well as a map $\bgamma\colon C(k)\to \HH^1(k,\Gamma)$ such that for $P\in C(k)$ and $\gamma=\bgamma(P)$, we have $Q\in D_\gamma(k)$ such that $\pi_\gamma(Q)=P$. This leads us to consider the associated Selmer set
\[\Sel^{(\pi)}(C/k)=\{\gamma\in \HH^1(k,\Gamma): D_\gamma(k_v)\neq \emptyset \text{ for all completions } k_v \text{ of }k\}.\]
Since the map $\bgamma$ takes values in $\Sel^{(\pi)}(C/k)$, we see that if the latter is empty then $C(k)$ is empty too. In that case we say that $C$ has a \emph{$\pi$-cover obstruction} to having rational points: $C$ has no rational points because a collection of covering varieties all have local obstructions.

The proof of the Chevalley-Weil theorem \cite{chevweil:covers} implies that $\Sel^{(\pi)}(C/k)\subset \HH^1(k,\Gamma;S)$, where the latter denotes the classes that are unramified outside the set $S$ of bad places for the cover $\pi\colon D\to C$. The set $\HH^1(k,\Gamma;S)$ is finite and explicitly computable. This means that to compute $\Sel^{(\pi)}(C/k)$ one only needs to check the local solvability of finitely many $D_\gamma$. Hence, $\Sel^{(\pi)}(C/k)$ is explicitly computable, although not necessarily efficiently.

For hyperelliptic curves, there is a well-developed theory of \emph{two-covers} \cite{BS:twocov}, where $\Gamma=\Jac_C[2]$. Their associated Selmer sets are relatively practical to compute and, as is described there, many genus two curves over $\QQ$ have no local obstruction, but can be shown to have $\Sel^{(2)}(C/\QQ)=\emptyset$. In fact it has since been shown \cite{BhargavaGrossWang2017} that in a precise way, \emph{most} hyperelliptic curves have a two-cover obstruction.

Results beyond hyperelliptic curves are sparse. The general descent theory is available in \cite{BPS:quartic}, which also provides some genus three examples, but in its full generality, the need to compute class group information of degree $28$ extensions limits large-scale experiments significantly. There is also some progress on creating an appropriate setting for arithmetic statistical techniques \cite{Thorne2016} to two-descent on Jacobians of curves of genus three, but it is presently not clear how to generalize the Bhargava--Gross--Wang approach to this setting.

In this article we endeavour to start a more systematic study by considering plane quartics $C$ with a restricted $2$-level structure; in particular $\Jac_C[2](\QQ)=(\ZZ/2\ZZ)^6$. This forces the $28$ bitangents of $C$ to be defined over $\QQ$ and has the computational and expository advantage that all required data can be expressed over $\QQ$; no algebraic number theory is required.

\begin{remark}
For a hyperelliptic curve $C$ of genus $g$, having $\Jac_C[2](k)=(\ZZ/2\ZZ)^{2g}$ implies that all $2g+2$ Weierstrass points on $C$ are rational, making two-cover descent rather uninteresting.  In this sense, two-cover descent on plane quartics has simpler non-trivial applications than on hyperelliptic curves.
\end{remark}

In Section~\ref{S:generating_quartics} we review an explicit description of the moduli space of smooth plane quartics with labelled bitangents as the space of seven labelled points in general position in $\PP^2$. For small fields we prove
\begin{prop}\label{P:small_primes}
	For $p=3,5,7$, there exist no nonsingular plane quartics over $\FF_p$ 
	with all bitangents defined over $\FF_p$.
	Over $\FF_9$, there is only one isomorphism class, represented by the Fermat quartic
	\[C_9\colon x^4+y^4+z^4=0, \text{ with }\#C_9(\FF_9)=28.\]
	Over $\FF_{11}$, there is only one isomorphism class, represented by:
	\[C_{11}\colon x^4+y^4+z^4+x^2y^2+x^2z^2 + y^2z^2=0, \text{ and } C_{11}(\FF_{11})=\emptyset.\]
\end{prop}
In particular,  a plane quartic $C$ over $\QQ$ with rational bitangents has bad reduction at $3$, $5$, and $7$. If it has good reduction at $11$, then it has a local obstruction there. The curve $C_9$ attains the maximum number of rational points for a genus three curve over $\FF_9$. Its rational points are contacts of the $28$ hyperflexes. Both $C_9$ and $C_{11}$ are reductions of the Klein quartic
$x^4+y^4+z^4-\tfrac{3}{2}(1+\sqrt{-7})(x^2y^2+x^2z^2+y^2z^2)$.

Section~\ref{S:twocover} describes, given a smooth plane quartic $C$ with rational bitangents, an explicit model for a two-cover $\pi_\gamma\colon D_\gamma\to C$, with $\Gamma=(\ZZ/2\ZZ)^6$ as a Galois-module. This directly establishes a description of two-covers and their twists, without appealing to \'etale cohomology.

In Section~\ref{S:selmersets} we describe an algorithm to compute,  sets $\Sel^{(2)}(C/k, N)\supset \Sel^{(2)}(C/k)$, for integers $N\geq 1$, with equality holding for $N\geq66569$, and, in practice, for much smaller values of $N$ already. The algorithm is reasonably efficient and can be applied in many practical situations.

In Section~\ref{S:results} we describe a numerical experiment, where we tabulate the behaviour of $\Sel^{(2)}(C/\QQ)$ for various quartics $C$. We consider a systematic collection of $81070$ moduli points with coordinates from $\{-6,\ldots,6\}$, as well as a collection of $70000$ randomly selected points with coordinates from $\{-40,\ldots,40\}$.

\begin{observation} For all curves $C$ in our collections with $\Sel^{(2)}(C/\QQ)\neq \emptyset$, we can find a point $P\in C(\QQ)$.
\end{observation}

This leaves the following question, which we fully expect to have an affirmative answer, but remains open for now.

\begin{question} Is it possible to
	construct a smooth plane quartic $C$ over $\QQ$ with rational bitangents such that $\Sel^{(2)}(C/\QQ)\neq \emptyset$ but $C(\QQ)=\emptyset$?
\end{question}

\begin{remark}
For a considerable number of curves in our collections we also get information on the $2$-Selmer groups of their Jacobians. The data 
matches the distribution conjectured in \cite{PoonenRains2012}*{Conjecture~1.1} quite closely, but only after taking into account that the $\Jac_C$-torsor representing $\Pic^1$ is very frequently everywhere locally trivial. Since non-hyperelliptic curves often have points everywhere locally, this phenomenon should be general: one should expect Jacobians to exhibit special arithmetic behaviour.
\end{remark}
This work is based on the master's thesis \cite{Lewis2019} of the second author.

\section{Plane quartics and their bitangents}
\label{S:bitangents}

In this section we collect the classical combinatorics and geometry of bitangents and theta characteristics on non-hyperelliptic curves of genus three. See \cite{CAG}*{Chapter~6} or \cite{GrossHarris2004} for a more comprehensive modern treatment.

Let $k$ be a field of characteristic different from $2$ and let $C$ be a curve of genus three over $k$. Then $\Jac(C)[2]$ is a $0$-dimensional separated group scheme of degree $64$ and exponent $2$, equipped with a non-degenerate alternating bilinear pairing. Indeed, the automorphism group of $\Jac(C)[2]$ is $\Sp_6(\FF_2)$.

\begin{defn} A \emph{theta characteristic} on a curve $C$ of genus $g$ is a divisor class $\theta\in\Pic^{g-1}(C)$ such that $2\theta$ is the canonical class. The \emph{parity} of $\theta$ is determined by the parity of the dimension of the Riemann-Roch space $\HH^0(C,\theta)$.
\end{defn}

It is a classical result \cite{GrossHarris2004}*{Proposition~1.11} that a curve of genus $g$ has $2^{g-1}(2^g+1)$ even and $2^{g-1}(2^g-1)$ odd theta characteristics. For $g=3$ and $C$ non-hyperelliptic it is easily checked that $h^0(C,\theta)\leq 1$, so the odd theta characteristics are exactly the ones that admit a (unique) effective representative.

The canonical model of a non-hyperelliptic genus three curve $C$ is a quartic in $\PP^2$:
\[C\colon f(x,y,z)=0, \text{ with }f\in k[x,y,z]\text{ homogeneous of degree four.}\] Since canonical classes are exactly line sections $C\cdot l$, we see there are $28$ lines $l$ such that $C\cdot l=2\theta$, where $\theta$ is a degree two effective divisor representing a theta characteristic: we recover the $28$ bitangents of a smooth plane quartic. Fix for each bitangent line $l$, a linear form $\ell$ describing the line.

\begin{lemma}\label{L:lin_indep_bitangents}
Let $C$ be a smooth plane quartic. Then no seven distinct bitangents pass through a single point.
\end{lemma}
\begin{proof}
Suppose $l_1,\ldots,l_7$ intersect in $P_0$. If $P_0$ were to lie on $C$ it would be singular, so it does not. Hence projecting away from $P_0$ gives a degree four map $C\to\PP^1$. Since $l_i\cdot C$ is a fiber of this projection, the ramification divisor has degree at least $2\cdot 7$. But that exceeds the degree $12$ given by Riemann-Hurwitz.
\end{proof}

Let $\theta_1,\theta_2$ be two odd theta-characteristics. Then $2(\theta_1-\theta_2)=\div(\ell_1/\ell_2)$, where we regard the quotient of linear forms as a rational function on $C$. We see that $[\theta_2-\theta_1]\in \Pic^0(C)[2]$. As it turns out, all nonzero $2$-torsion classes admit such a representative, in fact, $\binom{28}{2}/63=6$ of them. We see that $\theta_1-\theta_2$ and $\theta_3-\theta_4$ are linearly equivalent precisely when $\theta_1+\cdots+\theta_4$ is twice canonical. For bitangent forms, this leads to the following concept.
\begin{defn}\label{D:syzquad}
We say a quadruple of bitangent forms $\fq=\{\ell_1,\ldots,\ell_4\}$ is a \emph{syzygetic quadruple} if their contact points with $C$ lie on a conic. This means there are constants $\delta_\fq,c_\fq\in k^*$ and a quadratic form $Q_\fq\in k[x,y,z]$ such that
\begin{equation}\label{E:delta_equation}
\ell_1\ell_2\ell_3\ell_4=\delta_\fq Q_\fq^2 + c_\fq f.
\end{equation}
\end{defn}
There are 315 syzygetic quadruples. We say a triple of bitangents is \emph{syzygetic} if it is part of a syzygetic quadruple. If it is, then it is part of only one.
\begin{defn}
We say that a set of seven bitangent forms $\{\ell_1,\ldots,\ell_7\}$ is an \emph{Aronhold set} if none of its triples are syzygetic.
\end{defn}
There are $288$ Aronhold sets. For an Aronhold set, write $\{\theta_1,\ldots,\theta_7\}$ for the corresponding theta characteristics. Then $\theta_1+\cdots+\theta_7-3\kappa_C$ is again a theta characteristic: an even one. We see that each even theta characteristic has $288/36=8$ Aronhold sets associated with it. Additionally, one can check that $\{\theta_1-\theta_7,\ldots,\theta_6-\theta_7\}$ forms a basis for $\Pic(C)[2]$.

It follows that specifying a labelled Aronhold set on a smooth plane quartic amounts to marking a $2$-level structure on its Jacobian. The converse holds too.

\begin{prop}[\cite{GrossHarris2004}]
	The following two moduli spaces are naturally isomorphic.
	\begin{itemize}
		\item non-hyperelliptic genus three curves with a labelled Aronhold set
		\item non-hyperellipic genus three curves with full $2$-level structure.
	\end{itemize}
\end{prop}

There is a unique conjugacy class $\Sym(8)\subset \Sp_6(\FF_2)$. It is of length $36$ and it corresponds to the stabilizer of an even theta characteristic. The action can be made explicit by labelling the bitangents by
\begin{equation}\label{E:bitangent_labelling}
\{\ell_{ij}=\ell_{\{i,j\}}: i\in \{0,\ldots,7\}, j \in \{i+1,\ldots,7\}\},
\end{equation}
with $\Sym(8)$ acting in the obvious way on the subscripts. This labelling can be chosen in such a way that the syzygetic quadruples come in two $\Sym(8)$-orbits: one of length $210$ and one of length $105$, represented by, respectively,
\begin{equation}
\{\ell_{01},\ell_{12},\ell_{23},\ell_{03}\}\text{ and }
\{\ell_{01},\ell_{23},\ell_{45},\ell_{67}\}.\label{E:syz_quad}
\end{equation}
We see that for $i=0,\ldots,7$, we have the Aronhold sets $\{\ell_{ij}: j\neq i\}$. We sometimes suppress $i=0$ in our indices, so $\ell_{0j}=\ell_j$. 

\begin{prop}\label{P:aronhold_basis}
Let $\ell_1,\ldots,\ell_7$ be an Aronhold set of bitangent forms on a smooth plane quartic $C\colon f(x,y,z)=0$. Then the square class of each of the other bitangents $\ell_{ij}$ is determined in the sense that there is a constant $\delta_{ij}\in k^\times$ and a cubic form $g_{ij}\in k[x,y,z]$ such that
\[\Big(\prod_{n\notin\{i,j\}} \ell_n\Big)\ell_{ij} \equiv \delta_{ij}g_{ij}^2 \pmod{fk[x,y,z]}.\]
\end{prop}
\begin{proof}
To ease notation, set $\{i,j\}=\{6,7\}$. By combining the syzygetic quadruples
\[\{\ell_1,\ell_{23},\ell_{45},\ell_{67}\},
\{\ell_2,\ell_7,\ell_{23},\ell_{37}\},
\{\ell_4,\ell_7,\ell_{45},\ell_{57}\},
\{\ell_3,\ell_5,\ell_{37},\ell_{57}\} \]
we get that the left hand side has a divisor with even multiplicities. The existence of $g_{ij}$ follows from the projective normality of $C$.
\end{proof}

\section{Generating plane quartics with rational bitangents}
\label{S:generating_quartics}
We use del Pezzo surfaces of degree two (see \cite{CAG}*{6.3.3} or \cite{GrossHarris2004}) to describe a classical link between non-hyperelliptic genus three curves with $2$-level structure and point configurations in the plane.
\begin{defn}
We say seven points $p_1,\ldots,p_7\in\PP^2$ lie in \emph{general position} if no three are collinear and no six lie on a conic.
\end{defn}
Given seven points $p_1,\ldots,p_7\in\PP^2$ in general position, we obtain a del Pezzo surface $X$ of degree two by blowing up the seven points. In fact we obtain a labelling of the $56$ exceptional curves on $X$:
\begin{itemize}
	\item Seven exceptional components $E'_i$ above the blown-up points $p_i$
	\item Seven proper transforms $E_i$ of cubics $\tilde{E}_i$ through the seven points with a nodal singularity at $p_i$
	\item $21$ proper transforms $E_{ij}$ of lines $\tilde{E}_{ij}$ connecting $p_i$ and $p_j$.
	\item $21$ proper transforms $E'_{ij}$ of conics $\tilde{E}'_{ij}$ through $\{p_1,\ldots,p_7\}\setminus\{p_i,p_j\}$.
\end{itemize}
A del Pezzo surface $X$ of degree $2$ comes equipped with a $2:1$ 
map $X\to\PP^2$, given by the anticanonical system $|-\kappa_X|$ on $X$.
The branch locus $C$ in $\PP^2$ is a smooth plane quartic.

If $X$ is obtained as the blow-up of $p_1,\ldots,p_7\in\PP^2$ then there is an induced rational map $\phi$ making the following diagram commute.
\[
\begin{tikzcd}
&X \arrow[rd,"2:1"]\arrow[ld,"\mathrm{bl}"']\\
\PP^2\arrow[rr, dashrightarrow,"\phi"]&&\PP^2
\end{tikzcd}
\]
Let $\phi_1,\phi_2,\phi_3$ generate the space of cubics passing through $p_1,\ldots,p_7$. It is straightforward to check that the $\mathrm{bl}^*\phi_i$ generate $|-\kappa_X|$, so $\phi=(\phi_1:\phi_2:\phi_3)$. The branch locus of $\phi$ is contained in the plane sextic curve
\begin{equation}\label{E:plane_sextic_model}
C'\colon \det\left(\frac{\partial \phi_i}{\partial x_j}\right)_{ij}=0
\end{equation}
and indeed, $C=\phi(C')$ turns out to be a plane quartic.

Since $\tilde{E}_i$ and $\tilde{E}_{ij}\cup\tilde{E}'_{ij}$ are loci described by cubics in the span of $\phi_1,\phi_2,\phi_3$, they map to lines, whose defining forms we denote by $\ell_i$ and $\ell_{ij}$ respectively.
\begin{lemma}
The labelling described above is compatible with \eqref{E:bitangent_labelling}, so $\{\ell_1,\ldots,\ell_7\}$ is an Aronhold set and \eqref{D:syzquad} describes the syzygetic quadruples.
\end{lemma}
\begin{proof}
The deeper reason is that the configuration of seven points in $\PP^2$ has the same moduli as seven points in $\PP^3$ by \emph{association} of point sets \cite{Coble1922}. The sextic model $C'$ actually arises as the projection from a linear system $|\theta_\text{even}+\kappa_C|$ (see \cite{GrossHarris2004}), so the labelling is indeed directly linked to the choice of an even theta characteristic on $C$. However, it is also sufficient to just verify the statement for a particular case and then argue via connectedness of the moduli space.
\end{proof}

The construction above provides a very explicit description of the moduli space of non-hyperelliptic genus three curves with full $2$-level structure. For explicitly parametrizing it, we lose no generality by setting $p_1,p_2,p_3,p_4$ to be the standard simplex and choosing $p_5,p_6,p_7=(u_1:v_1:1),(u_2:v_2:1),(u_3:v_3:1)$. General position means the $3\times 3$, respectively $6\times 6$ minors of
\[\begin{pmatrix}
1&0&0&1&u_1&u_2&u_3\\
0&1&0&1&v_1&v_2&v_3\\
0&0&1&1&1&1&1
\end{pmatrix}\text{ and }\begin{pmatrix}
1&0&0&1&u_1^2&u_2^2&u_3^2\\
0&1&0&1&v_1^2&v_2^2&v_3^2\\
0&0&1&1&1&1&1\\
0&0&0&1&u_1v_1&u_2v_2&u_3v_3\\
0&0&0&1&u_1&u_2&u_3\\
0&0&0&1&v_1&v_2&v_3
\end{pmatrix},\]
do not vanish.

\begin{proof}[Proof of Proposition~\ref{P:small_primes}]
With the description given above, it is a finite amount of work to check all the possibilities for $p=3,5,7,11$. For $p=3,5,7$ there are no $7$ points over $\FF_p$ in general position (see also \cite{BanwaitFiteLoughran2019}*{Proposition~4.4}). For $\FF_9$ there are $40$ triples $\{(u_1:v_1:1),(u_2:v_2:1),(u_3:v_3:1)\}$ that complement the standard simplex to $7$ points in general position. The construction \eqref{E:plane_sextic_model} requires lifting to characteristic $0$, but the rest of the construction remains valid. We find all resulting curves are isomorphic to $C_9$.
For $\FF_{11}$ there are $1440$ triples, all giving curves isomorphic to $C_{11}$.
\end{proof}
\section{Two-covers of smooth plane quartics with rational bitangents}
\label{S:twocover}

Let $C\colon f(x,y,z)=0$ be a smooth plane quartic with an Aronhold set $\ell_1,\ldots,\ell_7$. We adopt the notation of Proposition~\ref{P:aronhold_basis}. For $\gamma=(\gamma_1,\ldots,\gamma_7)\in (k^\times)^7$
we define the following curve in weighted projective space $\PP[2^3,1^{28}]$ with coordinates $x,y,z$ of weight $2$ and $w_1,\ldots,w_7,w_{12},\ldots,w_{67}$ of weight $1$.
\[
D_\gamma'\colon
\left\{
\begin{aligned}
f(x,y,z)&=0,\\
\ell_i(x,y,z)&=\gamma_i w_i^2&\text{ for }i=1,\ldots,7,\\
\ell_{ij}(x,y,z)&=\frac{\delta_{ij}}{\prod_{n\neq i,j}\gamma_n}w_{ij}^2
  & \text{ for }1\leq i<j\leq 7,\\
g_{ij}(x,y,z)&=w_{ij}\textstyle\prod_{n\neq i,j} w_n
  & \text{ for }0\leq i<j\leq 7.\\
\end{aligned}
\right.
\]
Thanks to the relations from Proposition~\ref{P:aronhold_basis} we have a well-defined projection $D'_\gamma\to C$. In fact, from the sign changes on $w_1,\ldots,w_7$ we see that $\Aut(D_\gamma'/C)=(\ZZ/2\ZZ)^7$. Furthermore, from the fact that the representation of the automorphism group on $w_{12},w_{23},\ldots,w_{67},w_{17}$ is faithful and for any fiber of $D'_\gamma\to C$ at most one of $w_i$ or $w_{ij}$ is zero, it follows the cover is unramified and that $D'_\gamma$ is not geometrically connected. Indeed the involution on $D'_\gamma$ that swaps the signs of all of $w_1,\ldots,w_7$ interchanges geometric components.
We consider the projection $\PP[2^3,1^{28}]\to\PP^{27}$ away from the weight $2$ part and consider the image $D_\gamma$ of $D'_\gamma$.

Lemma~\ref{L:lin_indep_bitangents} yields three linearly independent linear forms $\ell_i,\ell_j,\ell_n$, so that we can express $x,y,z$ as linear forms in $w_i^2,w_j^2,w_k^2$. Eliminating $x,y,z$ from the equations gives us $D_\gamma$ as an intersection of an octic equation, $25$ quadratic equations, and $28$ sextic equations. Alternatively we derive quartic relations from the syzygetic quadruples and their described relations (see Definition~\ref{D:syzquad}).

We introduce notation for a group naturally isomorphic to $(k^\times/k^{\times2})^6$, but presented in a way more natural for our purposes.

\begin{defn}
We define $L'(2,k)\simeq (k^\times/k^{\times2})^6$ by the exact sequence
\[1\to (k^\times/k^{\times2})\stackrel{\text{diagonal}}\longrightarrow (k^\times/k^{\times2})^7 \to L'(2,k)\to 1.\]
and we usually represent elements in $L'(2,k)$ by $(\gamma_1,\ldots,\gamma_7)\in (k^\times)^7$.
\end{defn}

\begin{prop} The two-covers of $C$ are exactly
\[\{\pi_\gamma\colon D_\gamma\to C, \text{ where } \gamma\in L'(2,k)\}.\]
\end{prop}

\begin{proof}
The projection of $D_\gamma$ onto the coordinates $(w_1:\cdots:w_7)$ gives a birational map to an intersection $\tilde{D}_\gamma$ of four quadrics and an octic hypersurface. Its singular locus is the pull-back along $\pi_\gamma$ of the contact locus of the bitangents $\ell_1,\ldots,\ell_7$. 
We see that $\tilde{\pi}\colon \tilde{D}_\gamma\to C$ is a finite rational cover of degree $2^6$ and that $\tilde{\pi}^*(\ell_i/\ell_7)=(\gamma_i/\gamma_7) (w_i/w_7)^2$. This shows that a basis for $\Pic^0(C)[2]$ pulls back to principal divisors, and hence that $\tilde{D}_\gamma$ is a birational model of a two-cover, and therefore so is $D_\gamma$. To see that $D_\gamma$ is nonsingular, we use that for $P\in D_\gamma(\kbar)$ we can find an Aronhold set of bitangents that do not meet $\pi_\gamma(P)$.

In order to show that all $2$-covers arise as $D_\gamma$, we observe that $\Pic(C/k)[2]=(\mu_2)^6$, where we write $\mu_2$ for the Galois module $\{-1,1\}$. By the Kummer sequence we have
\[\HH^1(k,\Pic(C/k)[2])=(k^\times / k^{\times2})^6\simeq L'(2,k).\]
For $\sigma\in \Gal(k^{\sep}/k)$ we define the cocycle
\[\xi_\gamma(\sigma)\colon (w_1:\ldots:w_7)\mapsto \left(\frac{\sqrt{\gamma_1}^\sigma}{\sqrt{\gamma_1}}w_1:\cdots:
\frac{\sqrt{\gamma_7}^\sigma}{\sqrt{\gamma_7}}w_7\right).\]
This gives an isomorphism $L'(2,k) \simeq \HH^1(k,\Aut(D_1,C))\simeq \HH^1(k,\Pic^0(C)[2])$,
and $D_\gamma$ is the twist of $D_1$ by the Galois cocycle $\xi_\gamma$.
\end{proof}
We define a partial map
\[\bgamma\colon C(k)\dashrightarrow L'(2,k);\;P\mapsto(\ell_1(P),\ldots,\ell_7(P))\]
and extend it to a full map by observing that by Definition~\ref{D:syzquad}, for any syzygetic quadruple $\fq=\{\ell_i,\ell_a,\ell_b,\ell_c\}$ we have that
\[\ell_i(P)\equiv \delta_\fq \ell_a(P)\ell_b(P)\ell_c(P) \pmod{\text{squares}}\]
whenever both sides are nonzero, so if $\ell_i(P)=0$, we assign the appropriate value by taking the right hand side for a suitable quadruple $\fq$. We obtain
\begin{prop}\label{P:gamma_map}
	The map $\bgamma\colon C(k)\to L'(2,k)$ assigns to $P\in C(k)$ the cover $D_{\bgamma(P)}$ for which there is a point $Q\in D_{\bgamma(P)}(k)$ such that $\pi_{\bgamma(P)}(Q)=P$.
\end{prop}

\section{Selmer sets}
\label{S:selmersets}
We restrict to the case where $k$ is a number field, but our method applies to any global field of characteristic different from $2$. We write $\calO$ for its ring of integers, $\Omega$ for the set of places of $k$, and $k_v$ for the completion of $k$ at $v\in\Omega$. For non-archimedean $v$ we write $\calO_v\subset k_v$ for its ring of integers, $\fp_v$ for its maximal ideal, and $\calO_v/\fp_v$ for its residue field.

The map $\bgamma$ from Proposition~\ref{P:gamma_map} and its local variant $\bgamma_v$ fit in the commutative diagram
\[\begin{tikzcd}
C(k)\arrow[r,"\bgamma"]\arrow[d, hook]&L'(2,k)\ar[d,"\rho_v"]\\
C(k_v)\arrow[r,"\bgamma_v"]&L'(2,k_v)\nospaceperiod
\end{tikzcd}
\]
We define
\[\Sel^{(2)}(C/k)=\{\gamma\in L'(2,k): \rho_v(\gamma)\in \bgamma_v(C(k_v))\text{ for all }v\in\Omega_k\}.\]
Clearly we have $\bgamma(C(k))\subset \Sel^{(2)}(C/k)$ and in particular, if $\Sel^{(2)}(C/k)=\emptyset$ then $C(k)=\emptyset$.

Let us now fix an integral model $C\colon f(x,y,z)=0$ with $f\in \calO[x,y,z]$, as well as $28$ bitangent forms $\ell_{ij}\in\calO[x,y,z]$.
The \emph{discriminant} $D_{27}(f)$ of a quartic (see \cite{Gelfand-Kapranov-Zelevinsky2008}*{Chapter~13, Proposition~1.7}) is an integer form of degree $27$ in the coefficients of $f$ that vanishes precisely when $f$ describes a singular curve. Thus, if we take
\[S=\{v\in \Omega_k: \ord_v(2D_{27}(f))>0, \text{ or } \ell_{ij}\in\fp_v[x,y,z], \text{ or $v$ is archimedean}\}\]
then $C$ has good reduction at all $v$ not in $S$, meaning that the coefficient-wise reductions of $f$ and $\ell_{ij}$ describe a nonsingular plane quartic and its bitangents over $\calO_v/\fp_v$.
We consider the \emph{unramified part}
\[L'(2,k_v)^{\unr}=\{\gamma\in L'(2,k_v): \ord_v(\gamma_i)\equiv \ord_v(\gamma_j) \pmod{2}\text{ for all }i,j\}.\]

\begin{prop}\label{P:unramified_image}
If $C/k_v$ has good reduction as a plane quartic and the residue characteristic of $k_v$ is odd, then $\bgamma_v(C(k_v))\subset L'(2,k_v)^{\unr}$. If furthermore $\#\calO_v/\fp_v\geq 66562$ then
$\bgamma_v(C(k_v))=L'(2,k_v)^{\unr}$.
\end{prop}
\begin{proof} Let $\overline{C}$ be the reduction of $C$.
Any point $P\in C(k_v)$ reduces to a point $\overline{P}\in \overline{C}(\calO_v/\fp_v)$. Since the bitangents do not share contact points, we have $\ord_v(\ell_i(P))>0$ for at most one $i$. Let $\fq=\{\ell_i,\ell_a,\ell_b,\ell_c\}$ be a syzygetic quadruple. 
The good reduction properties imply that $\ord_v(\delta_\fq)=0$, in the notation of Definition~\ref{D:syzquad}. We see that $\ell_i(P)\ell_a(P)\ell_b(P)\ell_c(P)$ must have even valuation, but that implies that $\ord_v(\ell_i(P))$ is even.

For the second part, we observe that for $\gamma\in L'(2,k_v)^{\unr}$, the curve $D_\gamma$ has good reduction as well. This curve has genus $129$ and, writing $q=\#\calO_v/\fp_v$, the Hasse-Weil bounds give
\[\#\overline{D}_\gamma(\calO_v/\fp_v)\geq q+1-2\cdot 129\sqrt{q},\]
so if $q\geq 66562$, then there is a (necessarily smooth) point on $\overline{D}_\gamma$, so Hensel lifting gives a point in $D_\gamma(k_v)$. The image of that point on $C$ maps to $\gamma$.
\end{proof}

We define
\[L'(2,k; S)=\{\gamma\in L'(2,k): \rho_v(\gamma)\in L'(2,k_v)^{\unr}\text{ for all }v\in \Omega_k\setminus S\}.\]
Let $\calO_S$ be the ring obtained by inverting the primes of the finite places in $S$. If $\calO_S$ has odd ideal class number then $L'(2,k;S)$ is generated by $(\calO_S^\times/\calO_S^{\times 2})^7$, so it is a finite group. Note that by enlarging $S$, we can ensure that $\calO_S$ has odd class number.

It follows from Proposition~\ref{P:unramified_image} that $\Sel^{(2)}(C/k)\subset L'(2,k;S)$. Furthermore, if we set
\[T=S\cup\{v\in \Omega_k: \#\calO_v/\fp_v< 66562\},\]
then we obtain
\begin{equation}\label{E:selset_finite_intersection}
\Sel^{(2)}(C/k)=\{\gamma\in L'(2,k; S): \rho_v(\gamma)\in \bgamma_v(C(k_v)) \text{ for } v\in T\}.
\end{equation}
Hence, if we can compute generators for $\calO_S^\times$, which is a standard task in algebraic number theory, and compute $\bgamma_v(C(k_v))$ for finite and real $v$, then we can compute the Selmer set.

\subsection{Computing the local image for archimedean places}
\label{S:local_image_infinite}
For $k_v=\CC$ we have that $\CC^\times=\CC^{\times2}$ and $C(\CC)\neq \emptyset$, so there is nothing to compute: the local image is the whole (trivial) group $L'(2,\CC)$.

For $k=\RR$ we have that $\RR^\times/\RR^{\times 2}$ is represented by $\{\pm 1\}$. Furthermore, a smooth plane quartic $C/\RR$ with all bitangents defined over $\RR$ has four components \cite{GrossHarris1981}*{Proposition~5.1}, and the map $\bgamma\colon C(\RR)\to L'(2,\RR)\simeq \FF_2^6$ is continuous and therefore constant on components. In order to find $\bgamma(C(\RR))$ we only need to find points on each component and evaluate $\bgamma$ there. Each pair of components has four bitangents touching each, so these contact points must be real. The remaining four bitangents might have complex conjugate contact points.
Each pair of components is separated by a bitangent, so $\bgamma$ actually takes different values on the components: we know that $\#\bgamma(C(\RR))=4$.

Since we need to compute the bitangents anyway, we can use the real contact points to evaluate $\bgamma$. Once we have found four different images, we know we have determined the entire image. 

\subsection{Computing the local image for finite places}
\label{S:local_image_finite}
In this section, we take $k$ to be a local field with ring of integers $\calO$, uniformizer $\pi$ with $\fp=\pi\calO$, and a set $D$ of representatives of $\calO/\fp$.

We have $k^\times\simeq \ZZ\oplus \calO^\times$. The map $\mu\colon k^\times\to k^\times/k^{\times2}\simeq (\ZZ/2\ZZ)\oplus (\calO^\times/\calO^{\times2})$ is constant on sets of the form $x_0+\fp^{\ord(4)+1}$, with $x_0\in\calO^\times$, as can easily be checked from the fact that Newton iteration for finding the roots of $y^2-x_0$ amounts to iterating the map $y\mapsto \frac{1}{2}(y+\frac{x}{y})$, which converges for $y\in 1+2\fp$ if $\ord((x_0-1)/4)>1$.

We assume we have $f,\ell_{ij}\in \calO[x,y,z]$ representing a quartic curve $C\colon f(x,y,z)=0$ and its bitangents. Furthermore, we assume we have the $\delta_\fq$ from Definition~\ref{D:syzquad} for all syzygetic quadruples $\fq$, or at least the $210$ that involve $\ell_1,\ldots,\ell_7$.

Note that any $P\in C(k)$ admits a representative of one of the forms $(x_0:y_0:1),(x_0:1:\pi y_0),(1:\pi x_0:\pi y_0)$, with $x_0,y_0\in \calO$, so it is sufficient to restrict ourselves to $\calO$-valued points on affine plane quartics.

We say a set of the form $\calB=(x_0+\fp^e)\times (y_0+\fp^e)$ is a \emph{Hensel-liftable ball} for $f(x,y)=0$ if $0\in f(\calB)$ and $(0,0)\notin \nabla_{xy} f(\calB)$, with $\nabla_{xy}$ denoting the gradient. In that case, applying Newton iteration to any point in $\calB$ converges to an $\calO$-valued point of $f(x,y)=0$. It is a standard result that the $\calO$-valued points on a nonsingular curve can be covered with finitely many Hensel-liftable balls (see Algorithm~\ref{A:HenselBalls} in the Appendix).

In addition, we require that $\bgamma$ is constant on $\calB\cap C(k)$.
For this we use that the component $\bgamma_i(P)$ can be computed via either $\mu(\ell_i(P))$ or, for a syzygetic quadruple $\fq=\{\ell_i,\ell_a,\ell_b,\ell_c\}$, by
$\mu(\delta_\fq \ell_a(P)\ell_b(P)\ell_c(P))$.
Since bitangents do not share contact points, we see that for sufficiently small balls, at least one of the descriptions will be constant. We can then evaluate the map at a single representative. We start with a covering of Hensel-liftable balls and refine it as required. With Algorithm~\ref{A:LocalImage} (see Appendix) we find
\[
\begin{split}
\bgamma(C(k))=&\,\textsc{LocalImage}(f(x,y,1))\cup\textsc{LocalImage}(f(x,1,\pi y))\\
&\cup\, \textsc{LocalImage}(f(1,\pi x,\pi y)).
\end{split}
\]

\begin{remark}
The additional condition that $\bgamma$ be constant on our Hensel-liftable balls $\calB$ is surprisingly easily satisfied. In experiments with $\calO=\ZZ_p$, including for $p=2$, we find that refinement is only rarely required. 

This happens because there are many syzygetic quadruples: each $\ell_i$ is involved in $45$. Hence, if $P$ lies close to a zero of $\ell_i$, then there is likely a quadruple $\fq$ such that $P$ lies far away from the contact points of the other three bitangents.

This is in stark contrast with the hyperelliptic case, where the role of the bitangent contact points is played by the Weierstrass points. They are fewer in number, but there are also fewer relations between them, necessitating higher lifting.
\end{remark}

\subsection{Overcoming combinatorial explosion}
If $k$ is a number field, then we can compute $L'(2,k;S)$ and the algorithms from Sections~\ref{S:local_image_infinite} and \ref{S:local_image_finite} allow us to compute the local images, so using \eqref{E:selset_finite_intersection} we can compute $\Sel^{(2)}(C/k)$. However, as an $\FF_2$-vector space, we have $\dim_2 L'(2,k;S)=6(\#S)$, and $S$ tends to have considerable size. For instance, if $k=\QQ$ and $C$ has points everywhere locally, then Proposition~\ref{P:small_primes} yields that $\{2,3,5,7,11,\infty\}\subset S$, so $\#L'(\QQ,2;S)\geq 2^{36}$. Consequently, the point-wise iteration over $L'(k,2;S)$ that  \eqref{E:selset_finite_intersection} suggests, is usually practically infeasible. We use some linear algebra first.

We extend $\bgamma$ linearly to divisors, while also keeping track of the parity of the degree,
\[\btgamma\colon\Div(C)\to \FF_2\times L'(2,k);\quad \btgamma\Big(\sum n_P P\Big)=\Big(\sum n_P,\prod \bgamma(P)^{n_P}\Big),\]
(see \cite{BPS:quartic}*{\S6}). One finds that principal divisors lie in the kernel, so $\btgamma$ descends to a map on $\Pic(C/k)$.
We write $W_v=\langle \btgamma(C(k_v))\rangle$ for the $\FF_2$-span. We write $W_v^0$ for the kernel of the projection $W_v\to\FF_2$ on the first coordinate, and $W_v^1$ for its complement.

Given explicit representations for $L'(2,k;S)$ and $L'(2,k_v)$ as $\FF_2$-vector spaces, it is straightforward to find a description of $\tilde{\rho}_v\colon \FF_2\times L'(2,k;S)\to \FF_2\times L'(2,k_v)$ as a linear transformation. We immediately obtain
\begin{equation}\label{E:selset_linalgbound}
\Sel^{(2)}(C/k)\subset W_{C}^1:=\bigcap_{v\in S} \tilde{\rho}_v^{-1}(W_v^1),
\end{equation}
where the intersection on the right hand side is easily computed as an affine subset using standard linear algebra tools, even if $\#S\sim 100$.

On $\Pic^0(C/k_v)$ the kernel of $\btgamma_v$ is exactly $2\Pic^0(C/k_v)$. Furthermore, with the presence of a point $P_0\in C(k_v)$ we have that $\Pic^0(C/k_v)=\Jac_C(k_v)$, and since the latter is a compact $k_v$-Lie group we have
\begin{equation}\label{E:jac_dimbound}
\#( \Jac_C(k_v)/2\Jac_C(k_v)) = (\#\Jac_C[2](k_v))/|2|_v^3,
\end{equation}
where we normalize
\[
|2|_v=\begin{cases}
2&\text{if $v$ is a real place},\\
4&\text{if $v$ is a complex place},\\
(\#\calO_v/\fp_v)^{-\ord_v(2)}&\text{if $v$ is a finite place}.
\end{cases}
\]
\begin{lemma}\label{L:local_image_generated}
Suppose that $C$ is defined over a completion $\QQ_v$ of\, $\QQ$. If $\{P_0,\ldots,P_r\}\subset C(\QQ_v)$ are such that
\[\dim_2\langle \bgamma_v(P_i)-\bgamma_v(P_0): i = 1,\ldots,r\rangle =
\begin{cases}
3&\text{if }\QQ_v=\RR,\\
9&\text{if }\QQ_v=\QQ_2,\\
6&\text{otherwise},
\end{cases}\]
then $\btgamma_v(\Pic^0(C/\QQ_v))=W_v^0$ and $W_v=\langle \btgamma(P_0),\ldots,\btgamma(P_r)\rangle$.
\end{lemma}
\begin{proof}
We have $\#\Jac_C[2](\QQ_v)=64$, so the dimension bound is just \eqref{E:jac_dimbound}. Thus the condition is that the divisor classes $[P_1-P_0],\ldots,[P_r-P_0]$ generate $\Pic^0(C/\QQ_v)/2\Pic^0(C/\QQ_v)$.
The second statement follows simply from $W_v=W_v^0+\btgamma(P_0)$.
\end{proof}
This lemma provides us in many cases with a way to compute $W_v$ directly and quickly. An alternative is to determine $\btgamma_v(C(k_v))$ using the algorithm sketched in Section~\ref{S:local_image_finite}. This has a complexity proportional to the size of the residue field $\calO_v/\fp_v$, which is rather bad.

In many cases the $k_v$-valued contact points of the bitangents are already sufficient to generate $W_v$. In fact for real places this is always the case by the argument in Section~\ref{S:local_image_infinite}.

It may be the case that $\Pic^0(C/k_v)/2\Pic^0(C/k_v)$ really does need divisors with higher degree places in their support. In that case, if the residue field is small enough, we can compute $W_v$ via Section~\ref{S:local_image_finite} or we can search for these higher degree places and use $\langle\btgamma_v(P_0)\rangle+\btgamma_v(\Pic^0(C/k_v))$ as an upper bound for $W_v$ in \eqref{E:selset_linalgbound}.

\begin{remark}
	If Lemma~\ref{L:local_image_generated} applies to all $v\in S$ then we compute the $2$-Selmer group of $\Jac_C$ as well, via
	\[\Sel^{(2)}(\Jac_C/\QQ)=\bigcap_{v\in S} \tilde{\rho}_v(W_v^0),\]
	and in any case the right hand side gives a subgroup of the Selmer group, so we get a lower bound in all cases. See Section~\ref{S:rank_sha}.
\end{remark}

\subsection{Information at good primes}

Let $k_v$ be a local field of odd residue characteristic, with $q=\#(\calO_v/\fp_v)$. Then
\[\#L'(2,k_v)^{\unr}=64.\]
If $C/k_v$ has good reduction $\overline{C}$, then $\bgamma_v(P)$ is already determined by the reduction of $P$, so using the Hasse-Weil bounds, we obtain
\[\#\bgamma_v(C(k_v))\leq \#\overline{C}(\calO_v/\fp_v)\leq q+1+6\sqrt{q}.\]
If $q\leq 29$ then $\bgamma_v(C(k_v))\subsetneq L'(2,k_v)^{\unr}$, and even if $q$ is larger, it is quite likely that the local image is not the entire unramified set. Hence, for small residue class field, many of the two-covers $D_\gamma$ fail to have points locally, even at primes of good reduction. We see that in the intersection \eqref{E:selset_finite_intersection}, the primes of small norm actually impose significant conditions.

Because computing local images for primes of larger norm is expensive, we define a more easily computed set that contains $\Sel^{(2)}(C/k)$, by
\[
\begin{split}
\Sel^{(2)}(C/k;N)=\{\gamma\in L'(2,k;S): (1,\gamma)\in W_C^1 \text{ for }v\in S\text{ and }\hspace{3em}\\\rho_v(\gamma)\in \bgamma_v(C(k_v)) \text{ for $v$ such that }\#(\calO_v/\fp_v)\leq N\}.
\end{split}\]
We compute this set using Algorithm~\ref{A:TwoCoverDescent}. If the resulting set is empty, then $C(k)$ is empty.

\begin{algorithm}[t]
\caption{TwoCoverDescent}\label{A:TwoCoverDescent}
\Input{Quartic $f\in \calO[x,y,z]$ describing a nonsingular plane quartic $C$ with bitangent forms $\{\ell_{ij}\in\calO[x,y,z]: 0\leq i<j\leq 7\}$ and the $\delta_\fq$ according to Definition~\ref{D:syzquad}, and a norm bound $N$.}
\Output{$\Sel^{(2)}(C/k;N)$}
$S\gets \{v\in \Omega_k: \ord_v(2D_{27}(f))>0, \text{ or } \ell_{ij}\in\fp_v[x,y,z], \text{ or $v$ is archimedean}\}$\\
$W\gets \FF_2\times L'(2,k;S)$\label{line:W}\\
\For{$v\in S$}{
	$\calP\gets \{\btgamma_v(P)\in C(k_v): \ell_{ij}(P)=0 \text{ for some }i,j\}$\\
	\If{\upshape$\dim_2\langle P-Q:P,Q\in \calP\rangle$ equals the bound in Lemma~\ref{L:local_image_generated}}{
		$W_v\gets \langle \calP\rangle$
	}
	\Else{
		$W_v\gets \langle\btgamma_v(C(k_v))\rangle$\text{ as computed in Sections~\ref{S:local_image_infinite} and~\ref{S:local_image_finite}.}
	}
	$W\gets W\cap\rho_v^{-1}(W_v)$\\
}
$W^1\gets \{w\in W: w_1=1\}$, where $w_1$ is the image of $w$ in $\FF_2$ from Line~\ref{line:W}\\
\For{\upshape$v\in \Omega_k: v\text{ is finite and }\#(\calO_v/\fp_v)\leq N$}{
	$W^1\gets \{w\in W^1\colon \tilde{\rho}_v(w)\in \btgamma_v(C(k_v))\}$.
}
\Return{W}
\end{algorithm}

\section{Results}
\label{S:results}
We implemented Algorithm~\ref{A:TwoCoverDescent} for $k=\QQ$ in Magma~\cite{magma} and tested it on two sample sets:
\begin{itemize}
	\item[\textbf{A.}] Curves parameterized by
	\[\{(u_1,\ldots,v_3)\in\{-6,\ldots,6\}: u_1<u_2<u_3\text{ and }u_1<v_1\}.\]
	The inequalities normalize some of the permutations possible on the points that lead to isomorphic curves. We found $81070$ configurations in general position. However, because of the small values of the coefficients, there are many configurations with extra symmetries, so we find many isomorphic curves in the configurations. We find $33471$ distinct values for $D_{27}$, indicating that the collection contains many non-isomorphic curves as well.
    \item[\textbf{B.}] $70000$ curves with $u_1,\ldots,v_3$ chosen uniformly randomly from $\{-40,\ldots,40\}$, while discarding configurations that are not in general position. We originally found two quartics with matching $D_{27}$. Their configurations differed by a permutation, so the curves were isomorphic. We replaced one of them.
\end{itemize}
In each case, we used Magma's \texttt{MinimizeReducePlaneQuartic} to find a nicer plane model, with smaller discriminant. Since isomorphisms change $D_{27}$ by a $27^\text{th}$ power, it is easy to tell from discriminants when curves are not isomorphic.

Typical examples take less than 2 seconds to execute, with the quartic reduction step being one of the more expensive and less predictable steps. Occasional anomalies arise, where computation of a local image at a large prime is required. The whole experiment represents about 126 CPU hours of work.

\begin{example}
As a small, typical, example, take
\[
\begin{pmatrix}
u_1&u_2&u_3\\
v_1&v_2&v_3
\end{pmatrix}=
\begin{pmatrix}
 17& -7& -9\\
 35& 3& 9
\end{pmatrix}.
\]
We find
\[
\begin{split}
C&\colon
9x^4 - 60x^3y + 357x^2y^2 + 246xy^3 + 16y^4 - 42x^3z + 259x^2yz - 168xy^2z\\
&\quad- 141y^3z + 31x^2z^2 - 492xyz^2 + 207y^2z^2 + 42xz^3 - 27yz^3 + 9z^4=0
\end{split}\]
and $D_{27}(C)=2^{34}\cdot3^{20}\cdot5^{10}\cdot7^8\cdot11^2\cdot13^6\cdot17^4\cdot19^4\cdot29^2\cdot37^2\cdot41^2$. 
The curve $C$ has points everywhere locally.
We have $\dim_2 L'(2,\QQ;S)=72$ and $W_C=\bigcap_{v\in S} \tilde{\rho}_v^{-1}(W_v)$ has $\dim_2W_C=10$. We find that $W^1_C$ is non-empty, so it has $2^9$ elements. Computing $W^1_{C,T}=\{w\in W^1_C:\tilde{\rho}_v(w)\in\btgamma_v(C(k_v))\text{ for }v\in T\}$ is quite doable, for various sets $T$. We conclude that $C(\QQ)=\emptyset$ from, for example,
\[\Sel^{(2)}(C/\QQ)\subset W_{C,T}^1=\emptyset\text{ for }T=\{2,3,5\} \text{ or } \{
 31, 43, 47, 53, 71, 83\}.\]
Furthermore, from the data computed we can conclude that
\[\dim_2\Sel^{(2)}(\Jac_C/\QQ)=\dim_2 W_C^0=9,\]
so either $\Jac_C(\QQ)$ has free rank $3$ or $\Sha(\Jac_C/\QQ)[2]$ is non-trivial.
\end{example}

\subsection{Results of two-cover descent}
We executed Algorithm~\ref{A:TwoCoverDescent} on our samples, with $N=50$. This allowed us to determine the existence of rational points on each of the curves. We summarize our findings in Table~\ref{T:results}. 

\begin{table}
\def\arraystretch{1.3}
\[\begin{array}{c||c|c|c|c|c}
&C(\QQ_v)=\emptyset&\Sel^{(2)}(C/\QQ)=\emptyset&\parbox{8em}{\center rational bitangent contact point}&\parbox{6em}{\center other rational point}&\text{total}\\
\hline
\multirow{2}{*}{\bf A}&3654&42477&34025&4568&81070\\
& 4.5\%&52\%&42\%&5.6\%&100\%\\
\hline
\multirow{2}{*}{\bf B}&521&63926&4830&1244&70000\\
&0.7\%&91\%&6.9\%&1.8\%&100\%
\end{array}
\]
\caption{Two-cover descent results}\label{T:results}
\end{table}

When $\Sel^{(2)}(C/\QQ)\neq \emptyset$ and $C$ has no rational bitangent contact points (possibly a hyperflex), we search for a low-height nonsingular point using $\texttt{PointSearch}$ on either the sextic model \eqref{E:plane_sextic_model} or the plane quartic model we construct from it. These are the curves reported in the ``other rational point'' column.
For two curves we needed to search up to a height bound of $10^7$. 

Another interesting fact is that local obstructions are quite rare (having a local obstruction implies $\Sel^{(2)}(C/\QQ)=\emptyset$). Furthermore we only found $C(\QQ_p)=\emptyset$ for $p=2,11,23$, and only when $C$ has good reduction at those places. Proposition~\ref{P:small_primes} gives a partial explanation of this fact. This is quite contrary to the case of hyperelliptic curves, where local obstructions do tend to occur at primes of bad reduction.

\subsection{Information on rank and $\Sha$}\label{S:rank_sha}
We have
\[\Sel^{(2)}(\Jac_C/\QQ)= L'(\QQ,2;S)\cap \bigcap_{v\in S} \rho_v^{-1} \bgamma_v(\Pic^0(C/\QQ_v)).\]
Lemma~\ref{L:local_image_generated} gives a condition for when the sets on the right hand side are generated by differences of degree $1$ points. For a reasonable proportion of our curves, our data allows us to compute $\Sel^{(2)}(\Jac_C/\QQ)$. We list the results in Table~\ref{T:selmerranks}. In the rest of this section, we only consider these examples.
\begin{table}\label{T:selmerranks}
\[
\def\arraystretch{1.3}
\begin{array}{c|ccccccccc}
&6&7&8&9&10&11&12&13&\\
\hline
\textbf{A}&0.05\%&18.7\%&39.4\%&29.1\%&10.1\%&2.28\%&0.29\%&0.006\%&(n=31990)\\

\textbf{B}&0& 20.2\%&41.8\%&27.9\%& 8.71\%& 1.27\%&0.10\%& 0.006\%&(n=51685)\\
\end{array}
\]
\caption{Distribution of $\dim_2\Sel^{(2)}(\Jac_C/\QQ)$ where our data allowed its computation}
\end{table}

With $\Jac_C[2](\QQ)=(\ZZ/2\ZZ)^6$, we must have that the Selmer rank is at least $6$, but as one can see, the distribution has an average significantly higher than that. Part of that is explained by the fact that $C$, and hence the class $J^1\in H^1(k,\Jac_C)$ representing $\Pic^1(C/\QQ)$ is trivial everywhere locally. Since $C$ has quadratic points, we can pull the class back under the homomorphism $\Sel^{(2)}(\Jac_C/\QQ)\to H^1(k,\Jac_C)[2]$ and the preimage is likely independent of  the image of $\Jac_C[2](\QQ)$.

If $W^1_C=\emptyset$ in \eqref{E:selset_linalgbound} then it follows by \cite{Creutz2020}*{Theorem~5.3} that $J^1$ is not divisible by two in $\Sha(\Jac_C/\QQ)$, and therefore is nontrivial. This happens in about half the examples. 

Once we take into account that we expect that $\dim_{2}\Sel^{(2)}(\Jac_C/\QQ)\geq 7$, we find that the distributions in Table~\ref{T:selmerranks}, particularly for collection~\textbf{B}, match \cite{PoonenRains2012}*{Conjecture~1.1} rather well. This does require us to account for the fact that $J^1$ almost always has points everywhere locally. 

Generally, non-hyperelliptic curves tend to have points everywhere locally. Therefore, one actually should expect that Selmer groups of Jacobians of curves behave a little differently from those of general abelian varieties, because they tend to come equipped with an everywhere locally trivial torsor.

\section*{Acknowledgments}
We thank Michael Stoll for interesting discussions and suggestions on how to interpret the rank results in light of \cite{PoonenRains2012}, and an anonymous referee for helpful comments.

\nocite{bruin:ternary} %needs to appear here because it's referenced in the appendix.
%%%%%%%%%%%%%%%%%%%%%%%%%%%%%%%%%%%%%%%%%%%%%%%%%%%%

%%%%%%%%%%%%%%%%%%%%%%%%%%%%%%%%%%%%%%%%%%%%%%%%%%%%
\newpage
\appendix
\section{Local algorithms}
We use the notation from Section~\ref{S:local_image_finite}. The algorithms here are in the spirit of \cite{bruin:ternary}*{\S5} and \cite{BS:twocov}*{\S4}.

\begin{algorithm}
\caption{\textsc{HenselBalls}}\label{A:HenselBalls}
\Input{$f\in \calO[x,y]$, describing a smooth curve.}
\Output{A finite set $\{(x_t,y_t,e_t)\}_t$ of Hensel-liftable balls covering the $\calO$-valued solutions of $f(x,y)=0$.}
\For{$(x_0,y_0)\in \{(x_0,y_0)\in D^2: f(x_0,y_0)\equiv 0\pmod{\fp}\}$}{
	$R \gets \emptyset$\\
    \If{\upshape$\frac{\partial f}{\partial x}(x_0,y_0)\not\equiv 0\pmod{\fp}$ or $\frac{\partial f}{\partial y}(x_0,y_0)\not\equiv 0\pmod{\fp}$}{
    $R\gets R\cup \{(x_0,y_0,1)\}$.
}
    \Else{
$g\gets f(x_0+\pi x,y_0+\pi y)$\\
$T\gets \textsc{HenselBalls}(g/\mathrm{content}(g))$\\
$R\gets R\cup \{(x_0+\pi x_1,y_0+\pi y_1,e+1): (x_1,y_1,e)\in T\}$.
}
\Return $R$
}\end{algorithm}

\begin{algorithm}
	\caption{\textsc{LocalImage}}\label{A:LocalImage}
	\Input{$f\in\calO[x,y]$ describing a smooth plane quartic, together with its bitangent forms $\{\ell_{ij}\in \calO[x,y]: 0\leq i<j\leq 7\}$ and syzygetic data $\delta_\fq$ as in Definition~\ref{D:syzquad}.}
	\Output{Local image of $\bgamma_v$ on the given affine patch}
	Denote the mod-squares map by $\mu\colon \calO\setminus\{0\}\to k^\times/k^{\times 2}$.\\
	$T\gets \mathrm{\textsc{HenselBalls}}(f)$\\
	$R\gets\emptyset$\\
	\While{$T\neq \emptyset$}{
		Take $(x_0,y_0,e)$ from $T$,\\
		$L\gets [\ell_{ij}(x_0,y_0): 0\leq i<j\leq 7]$\\
		\For{$i=1,\ldots,7$}{
			\If{$\ord(L_i)<e-\ord(4)$}{
				$\gamma_i\gets \mu(L_i)$
			}
			\ElseIf{\upshape there is a syzygetic quadruple $\fq=\{\ell_i,\ell_a,\ell_b,\ell_c\}$ such that ${}
			\max(\ord(\ell_a(x_0,y_0)),\ord(\ell_b(x_0,y_0)),\ord(\ell_c(x_0,y_0)))<e-\ord(4)$}{
				$\gamma_i\gets \mu(\delta_\fq\ell_a(x_0,y_0)\ell_b(x_0,y_0)\ell_c(x_0,y_0))$
			}
			\Else(\hfill[\textbf{Remark:} we refine the covering]){
				$g\gets f(x_0+\pi^e x,y_0+\pi^e y)$\\
				$h\gets g/\mathrm{content}(g)$\hfill[\textbf{Remark:} $h\pmod{\fp}$ will be linear]\\
				\For{$(x_1,y_1)\in \{(x_1,y_1)\in D^2 : h(x_1,y_1)\equiv 0 \pmod \fp\}$}{
					$T\gets T\cup (x_0+\pi^e x_1,y_0+\pi^e y_1,e+1)$}

				\textbf{break} to while	
			}
		}
		Add $(\gamma_1,\ldots,\gamma_7)$ to $R$.
	}
	\Return $R$.
\end{algorithm}

\end{document}